\documentclass[a4paper,12pt,bold,center]{article}
\usepackage[a4paper,margin=2.54cm]{geometry}

\usepackage[oldstylemath,fulloldstylenums]{kpfonts}
\usepackage[nocompress]{cite}
\usepackage[bf,big,center]{titlesec}

\titleformat{\section}[hang]{\bfseries\Large\filcenter}{\arabic{section}.}{0.5em}{}

\usepackage[hidelinks]{hyperref}
\usepackage{amsmath, amssymb, amsfonts}
\usepackage{amsthm, thmtools}
\usepackage[nameinlink,capitalise]{cleveref}

\declaretheorem[numberwithin=section]{corollary}

\declaretheorem[numberlike=corollary]{lemma}

\declaretheorem[numbered=no,style=remark,qed=$\blacksquare$]{proof}

\DeclareMathOperator{\ob}{ob}
\DeclareMathOperator{\id}{id}

\usepackage{mathrsfs}
\usepackage{xspace}
\def\bN{\ensuremath{\mathbb N}\xspace}

\def\sC{\ensuremath{\mathscr C}\xspace}

\newcommand{\prodcone}[5]{\ensuremath{{#1}\overset{#2}\leftarrow{#3}\overset{#4}\to{#5}}}

\newcommand{\pair}[1]{\ensuremath{\left<{#1}\right>}}

\usepackage{tikz}
\usetikzlibrary{calc,positioning}
\tikzset{node distance=0.5cm, auto, minimum size=0cm,
  baseline=(current bounding box.center),
  la/.style={scale=0.8},
  a/.style={->},
  d/.style={double, double equal sign distance,-},
  u/.style={dashed,->},
}

\title{A common misinterpretation of Isbell's obstruction to monoidal strictification}
\author{tslil clingman}
\date{}

\begin{document}

\maketitle

\abstract{A monoidal category has a natural isomorphism $\alpha_{A,B,C}\colon(A\otimes B)\otimes C\to A\times (B\otimes C)$ called the associator. In the case where the objects $(A\otimes B)\otimes C$ and $A\otimes(B\otimes C)$ are equal, it is natural to ask whether this map may be taken to be the identity. Isbell \cite{Isbell} gave an argument for an obstruction to strictifying the component of the associator of a cartesian monoidal category at an object $C=C\times C$. This argument has been widely reproduced and is commonly misunderstood as demonstrating that naturality of the associator is the true obstruction to strictification. We consider the hidden hypothesis in this argument, give a new argument not dependent on naturality but on the hidden hypothesis, and finally show that naturality alone is not the issue  -- rather the crux of Isbell's argument involves a hidden assumption concerning the product cones. Through this analysis we also resolve that there can be no general obstruction to strictifying a component of the of the associator, even at such an object $C$, in a cartesian monoidal category.}

\section{Introduction}

In \cite{Isbell}, Isbell wrote
\begin{quote}
  ``Any category having finite products has a coherently associative product functor [6]. But a skeletal category having infinite products cannot have a strictly associative binary product functor $\times$ unless it is schlicht\footnote{posetal}. Suppose \sC is not schlicht, but $\times$ is associative. For some objects $A,B$ there are two morphisms $A\to B$. Then an infinite power $C$ of $B$ satisfies $C\times C=C$ and has (accordingly) two endomorphic coordinate projections $p_{1}$,$p_{2}$. For endomorphisms of $C$, $f\times g$ has coordinates $p_{1}(f\times g)=fp_{1}$, $p_{2}(f\times g)=gp_{2}$. Then $p_{1}(f\times(g\times h))=fp_{1}=p_{1}((f\times g)\times h)=(f\times g)p_{1}$. Since $p_{1}$ is epic, $f=f\times g$ which is absurd.''
\end{quote}

Variations on this argument have been reproduced, for example, in \cite[\S VII.I]{MacLane}, \cite[Example 3.6.7]{Riehl}, and \cite[Theorem 10.1]{Braunling}. Where this argument appears it is typically used to assert that it is the naturality of the associator with a strictified component which is problematic. As we shall see however, this is not the case: we can always arrange for cartesian monoidal structures which validate all of the \emph{explicit} hypotheses of this argument, but not its conclusion. Nevertheless, under an additional assumption and by means of a different proof we may reproduce the claimed obstruction.

As suggested by our emphasis of the word ``explicit'', the \emph{raison d'\^etre} of the present paper is the presence of a hidden subtlety to this argument that we now expose, viz., the unaddressed choice of cones. If \sC is a category with all binary products and a terminal object, then in constructing a cartesian monoidal structure on \sC we must choose for each pair of objects $A,B\in\ob\sC$ not only a product object $A\times B$ but also a product cone
$\prodcone{A}{\pi_{A}}{A\times B}{\pi_{B}}{B}$. Unlike in a general monoidal category where the associator is data, in a cartesian monoidal category the associator $\alpha_{A,B,C}\colon (A\times B)\times C\to A\times (B\times C)$ is uniquely determined by a choice of, \emph{a priori}, four product cones.\footnote{As we shall see in \cref{lemma:paste_sinister,lemma:paste_dexter}, given a binary product $A\times B$ and cone, triple product cones for $A,B,C$ are in bijection with binary products cones for $(A\times B),C$. Indeed then, there are precisely four choices to be made.} That is, having chosen the below product cones, the associator $\alpha_{A,B,C}$ is the unique morphism satisfying the equations there displayed.
\begin{equation}\label{eqn:generic_alpha}
  \begin{gathered}
  \begin{tikzpicture}
    \node(l1)[]{$(A\times B)\times C$};
    \node(lb1)[below= of l1,xshift=-1cm]{$A\times B$};
    \node(lb2)[below= of lb1,xshift=-0.5cm]{$A$};
    \node(lb2p)[below= of lb1,xshift=0.5cm]{$B$};
    \node(lb3)[below= of l1,xshift=1cm]{$C$};
    \draw[a](l1)to node[la,swap,pos=0.8]{$\pi^{\ell}_{1}$}(lb1);
    \draw[a](l1) to node[la,pos=0.8]{$\pi^{\ell}_{2}$}(lb3);
    \draw[a](lb1)to node[la,swap,pos=0.9]{$\pi^{A\times B}_{1}$}(lb2);
    \draw[a](lb1) to node[la,pos=0.9]{$\pi^{A\times B}_{2}$}(lb2p);
    \coordinate (tl) at ($(lb3)+(4cm,0)$);
    \node(r1) at (tl|-l1){$A\times(B\times C)$};
    \node(rb1)[below= of r1,xshift=-1cm]{$A$};
    \node(rb3)[below= of r1,xshift=1cm]{$B\times C$};
    \node(rb2)[below= of rb3,xshift=-0.5cm]{$B$};
    \node(rb2p)[below= of rb3,xshift=0.5cm]{$C$};
    \draw[a](r1)to node[la,swap,pos=0.8]{$\pi^{r}_{1}$}(rb1);
    \draw[a](r1) to node[la,pos=0.8]{$\pi^{r}_{2}$}(rb3);
    \draw[a](rb3)to node[la,swap,pos=0.9]{$\pi^{B\times C}_{1}$}(rb2);
    \draw[a](rb3) to node[la,pos=0.9]{$\pi^{B\times C}_{2}$}(rb2p);
  \end{tikzpicture}\\\\
    \pi^{r}_{1}\alpha_{A,B,C}=\pi^{A\times B}_{1}\pi^{\ell}_{1}\ ,\quad
    \pi^{B\times C}_{1}\pi^{r}_{2}\alpha_{A,B,C}=\pi^{A\times B}_{2}\pi^{\ell}_{1}\ ,\quad
    \pi^{B\times C}_{2}\pi^{r}_{2}\alpha_{A,B,C}=\pi^{\ell}_{2}
  \end{gathered}
\end{equation}

With this in mind, in \cref{sec:explicit_isbell} we elaborate Isbell's argument and show its sensitivity the suppressed selection of product cones. In \cref{sec:samecones} we show that this hidden hypothesis is enough to deduce an obstruction without relying on the naturality of the associator. Thus in \cref{sec:differentcones} we are emboldened to demonstrate that one may always choose cones carefully so as to realise the associator component $\alpha_{A,B,C}$ as strict at a single but arbitrary triple of objects $A,B,C$ -- even if all three objects are the same $A=B=C$ and when $C$ is a product $C\times C$. We will then conclude in \cref{sec:conclusion} that Isbell's argument is therefore dependent upon the concealed choice of cones and that the obstruction to strictification does not lie in the naturality of the associator -- thereby refuting the common {misinterpretation} of Isbell's argument.

\section{Isbell's hidden assumption}\label{sec:explicit_isbell}

Let us begin by carefully examining Isbell's argument. The core of the proof in \cite{Isbell} makes explicit the assumption of an object $C$ which is a product $C=C\times C$ with a product cone ${p_{1}},{p_{2}}\colon C\to C$, as well as the presence of two endomorphisms $f,g\colon C\to C$.

Implicit in this proof is a choice of triple product cones for the object $C$. Let us be therefore unbiased in our choice and write generically the following for the triple product cones. Note that we have already fixed the product cone $C\leftarrow C\to C$ so that the ``bottom layer'' of the cones below is determined, but we are as yet unconstrained in the choice of cones above.
\begin{equation}\label{eqn:unbiased_isbell}
  \begin{tikzpicture}[node distance=0.3cm]
    \node(l1)[]{$C$};
    \node(lb1)[below= of l1,xshift=-1cm]{$C$};
    \node(lb2)[below= of lb1,xshift=-0.5cm]{$C$};
    \node(lb2p)[below= of lb1,xshift=0.5cm]{$C$};
    \node(lb3)[below= of l1,xshift=1cm]{$C$};
    \node[below=1.75cm of l1]{``$(C\times C)\times C$''};
    \draw[a](l1)to node[la,swap]{$p^{\ell}_{1}$}(lb1);
    \draw[a](lb1)to node[la,swap]{$p_{1}$}(lb2);
    \draw[a](lb1) to node[la]{$p_{2}$}(lb2p);
    \draw[a](l1) to node[la]{$p^{\ell}_{2}$}(lb3);
  \end{tikzpicture}\hspace{3cm}
  \begin{tikzpicture}[node distance=0.3cm]
    \node(r1)[]{$C$};
    \node(rb1)[below= of l1,xshift=-1cm]{$C$};
    \node(rb3)[below= of l1,xshift=1cm]{$C$};
    \node(rb2)[below= of rb3,xshift=-0.5cm]{$C$};
    \node(rb2p)[below= of rb3,xshift=0.5cm]{$C$};
    \node[below=1.75cm of l1]{``$C\times (C\times C)$''};
    \draw[a](r1)to node[la,swap]{$p^{r}_{1}$}(rb1);
    \draw[a](rb3)to node[la,swap]{$p_{1}$}(rb2);
    \draw[a](rb3) to node[la]{$p_{2}$}(rb2p);
    \draw[a](r1) to node[la]{$p^{r}_{2}$}(rb3);
  \end{tikzpicture}
\end{equation}

From these cones we construct the associator component $\alpha_{C,C,C}\colon C\to C$ as in \eqref{eqn:generic_alpha}, and Isbell then introduces the final hypothesis that this  associator component is the identity $\id_{C}$.

With the hypotheses fixed, we may turn our attention to constructing the various morphisms used in the argument. First we construct $f\times g\colon C\to C$ as the pairing $\pair{fp_{1},gp_{2}}_{p_{1},p_{2}}$ with respect to the cone $p_{1},p_{2}\colon C\to C$. Next the morphisms $(f\times g)\times h$ and $f\times (g\times h)$ are constructed as follows.
\begin{equation}\label{eqn:pairing_formulas}
  \begin{aligned}
    (f\times g)\times h&\coloneqq\pair{\pair{fp_{1},gp_{2}}_{p_{1},p_{2}}p_{1}^{\ell},\ hp_{2}^{\ell}}_{p_{1}^{\ell},p_{2}^{\ell}}\\
    f\times(g\times h)&\coloneqq\pair{fp_{1}^{r},\ \pair{gp_{1},hp_{2}}_{p_{1},p_{2}}p_{2}^{r}}_{p_{1}^{r},p_{2}^{r}}
  \end{aligned}
\end{equation}
Now let us examine Isbell's string of equalities, using $p_{1}^{\ell},p_{2}^{\ell}$ and $p_{1}^{r},p_{2}^{r}$ where necessary. First Isbell gives $p^{r}_{1}(f\times(g\times h))=fp_{1}^{r}$ which holds by definition from \eqref{eqn:pairing_formulas}. Next he uses the \emph{naturality} of the associator $\alpha_{C,C,C}$ as well as the hypothesis that  $\alpha_{C,C,C}=\id_{C}$ to deduce that  $fp_{1}^{r}=p_{1}^{r}((f\times g)\times h)$. His next step, however, is not generally valid.

Isbell writes $p^{r}_{1}((f\times g)\times h)=(f\times g)p^{r}_{1}$, but there is no reason that this should hold. Indeed, looking at \eqref{eqn:pairing_formulas} we see that $p^{r}_{1}((f\times g)\times h)$ is \emph{not} a quantity we have any ability to compute at all -- $(f\times g)\times h$ was defined in terms of the cone $(p^{\ell}_{1},p^{\ell}_{2})$! Here thus is Isbell's hidden hypothesis: \emph{if} we additionally choose the cones in \eqref{eqn:unbiased_isbell} so that $p^{\ell}_{1}=p^{r}_{1}=p_{1}$ and $p^{\ell}_{2}=p^{r}_{2}=p_{2}$ \emph{then} the argument goes through as stated and we may conclude that $f=f\times g$. In fact, \cref{cor:consistent} below shows that this assumption is necessary.

In the next two sections we will draw attention to the two features of the above proof we have emphasised: the choice of cones and the naturality of the associator.

\section{Strict associativity can be an obstruction}\label{sec:samecones}

In what follows we will aim to show that the naturality of the associator $\alpha$ is inessential to Isbell's conclusion, and instead the obstruction to strictification may be derived only from $\alpha_{C,C,C}=\id_{C}$ and the hidden hypothesis on cones.

Let us, as Isbell does, suppose the existence of an object $C$ with a product cone
$\prodcone C{p_{1}} C {p_{2}} C$ so that $C$ is a product of $C$ with $C$. To leverage this assumption we will choose the object $C$ for the products $(C\times C)\times C$ and $C\times(C\times C)$, but crucially also choose the following cones for these triple products.
\begin{equation}\label{eqn:samecones}
  \begin{tikzpicture}[node distance=0.3cm]
    \node(l1)[]{$C$};
    \node(lb1)[below= of l1,xshift=-1cm]{$C$};
    \node(lb2)[below= of lb1,xshift=-0.5cm]{$C$};
    \node(lb2p)[below= of lb1,xshift=0.5cm]{$C$};
    \node(lb3)[below= of l1,xshift=1cm]{$C$};
    \node[below=1.75cm of l1]{``$(C\times C)\times C$''};
    \draw[a](l1)to node[la,swap]{$p_{1}$}(lb1);
    \draw[a](lb1)to node[la,swap]{$p_{1}$}(lb2);
    \draw[a](lb1) to node[la]{$p_{2}$}(lb2p);
    \draw[a](l1) to node[la]{$p_{2}$}(lb3);
  \end{tikzpicture}\hspace{3cm}
  \begin{tikzpicture}[node distance=0.3cm]
    \node(r1)[]{$C$};
    \node(rb1)[below= of l1,xshift=-1cm]{$C$};
    \node(rb3)[below= of l1,xshift=1cm]{$C$};
    \node(rb2)[below= of rb3,xshift=-0.5cm]{$C$};
    \node(rb2p)[below= of rb3,xshift=0.5cm]{$C$};
    \node[below=1.75cm of l1]{``$C\times (C\times C)$''};
    \draw[a](r1)to node[la,swap]{$p_{1}$}(rb1);
    \draw[a](rb3)to node[la,swap]{$p_{1}$}(rb2);
    \draw[a](rb3) to node[la]{$p_{2}$}(rb2p);
    \draw[a](r1) to node[la]{$p_{2}$}(rb3);
  \end{tikzpicture}
\end{equation}
We wish to emphasise that this is merely a possible choice, it just so happens that we may reuse the cone $(p_{1},p_{2})$ when choosing all four cones for both triple products. There are, in general, many choices one might make for the triple product cones, each of which leads to a cartesian monoidal structure. However, having first made the same choice for every cone as in \eqref{eqn:samecones} above, there \emph{are} obstructions to the strictness of the associator at $C$. One such is given by the following lemma.

\begin{lemma}\label{lemma:subterm}
  Let \sC be a category, and let $C\in\ob\sC$ be an object equipped with a product cone $\prodcone C{p_{1}} C{p_{2}} C$. The unique morphism $\alpha_{C,C,C}\colon C\to C$ satisfying
  \begin{equation}\label{eqn:alpha}
    p_{1}\alpha_{C,C,C}=p_{1}p_{1},\quad\quad p_{1}p_{2}\alpha_{C,C,C}=p_{2}p_{1},\quad\text{and}\quad p_{2}\alpha_{C,C,C}=p_{2}p_{2}
  \end{equation}
  is the identity if and only if $\sC(-,C)$ is sub-terminal in $\widehat\sC={\normalfont\textsf{Cat}}(\sC^{\operatorname{op}},{\normalfont{\textsf{Set}}})$.
\end{lemma}

\begin{proof}
  One direction is clear: if $\sC(-,C)$ is sub-terminal then indeed $\alpha_{C,C,C}=\id_{C}$ is forced. Let us suppose now that $\alpha_{C,C,C}$ is the identity. By using the universal property of the product $C$ of $C$ with itself we may derive a section $s_{1}$ of $p_{1}$ as displayed below-left. But then $p_{1}=p_{1}p_{1}s_{1}=p_{1}s_{1}=\id_{C}$. Similarly we may deduce that $p_{2}=\id_{C}$ and so by the universal property displayed below-right, we see that any two morphisms $f,g\colon A\to C$ must be equal, whence $\sC(-,C)$ is sub-terminal.

  \hfill\begin{tikzpicture}[node distance=0.75cm,baseline=(current bounding box.south)]
    \node(1)[]{$C$};
    \node(2)[right= of 1]{$C$};
    \node(3)[right= of 2]{$C$};
    \node(4)[below= of 2]{$C$};
    \draw[a](2)to node[la,above]{$p_{1}$}(1);
    \draw[a](2)to node[la,above]{$p_{2}$}(3);
    \draw[a](4)to node[la,swap]{$p_{2}$}(3);
    \draw[d](4)to(1);
    \draw[u](4)to node[la]{$s_{1}$}(2);
  \end{tikzpicture}\hspace{2cm}
  \begin{tikzpicture}[node distance=0.75cm,baseline=(current bounding box.south)]
    \node(1)[]{$C$};
    \node(2)[right= of 1]{$C$};
    \node(3)[right= of 2]{$C$};
    \node(4)[below= of 2]{$A$};
    \draw[d](2)to(1);
    \draw[d](2)to(3);
    \draw[a](4)to node[la,swap]{$g$}(3);
    \draw[a](4)to node[la]{$f$}(1);
    \draw[u](4)to(2);
  \end{tikzpicture}\hfill
\end{proof}

Although we have not mentioned naturality of $\alpha_{C,C,C}$ with respect to the object $C$ at all, this lemma is enough for us to recover the conclusion of Isbell's argument under the additional hypothesis on cones.

\begin{corollary}[Isbell's obstruction]
  Let \sC be a category with binary products, and infinite powers of objects. Assume further that for every object $C$ which is a product $C\times C$ the associator was induced canonically by the choice of cones as in \eqref{eqn:samecones}. Then the component $\alpha_{C,C,C}\colon C\to C$ of the associator satisfies \eqref{eqn:alpha}, and is the identity if and only if \sC is posetal.\hfill$\blacksquare$
\end{corollary}

We may also re-derive, for example, the conclusion asserted in \cite[VII.1]{MacLane}. In the category {\normalfont\textsf{Set}} of sets and functions, \bN is a product of \bN with \bN. Then,

\begin{corollary}[Obstruction in \textsf{Set}]
  In {\normalfont\textsf{Set}}, the associator $\alpha_{\bN,\bN,\bN}\colon\bN\to\bN$ defined via \eqref{eqn:alpha} is the identity if and only if $0=1$.\hfill$\blacksquare$
\end{corollary}

We have phrased these obstructions in this manner in the hopes of convincing the reader that they are, in fact, unremarkable. Far from prohibiting any form of strictification at an object $C=C\times C$, these obstructions now merely stand to inform a more careful choice of cones for triple products of such objects. Indeed, even in Isbell's original argument we must crucially make use of the fact that we have chosen cones as in \eqref{eqn:samecones}, as we will now show.

\section{Strict associativity is consistent}\label{sec:differentcones}

Now that we have proven that a particular choice of cones for the triple product of objects $C=C\times C$ presents an obstruction to strictification of the associator, let us divert our attention from this particular case and give a general construction which informs a more compatible choice of cones.

While the notions of pasting and cancellation of pullback squares are well-documented in the literature, we record here the perhaps less publicised analogue for products.

\begin{lemma}\label{lemma:paste_sinister}
  Given objects $A,B,C\in\ob\sC$ and a product cone $\prodcone A {p_{1}} {A\times B}{p_{2}} B$, composition with this cone gives a bijection between the following sets.
  \[
    \left\{
      \begin{gathered}
        \textrm{\,product cones\,}\\
        \begin{tikzpicture}
          \node(1)[]{$P$};
          \node(2)[below left= of 1,anchor=north]{$A\times B$};
          \node(4)[below right= of 1,anchor=north]{$C$};
          \draw[a](1)to node[la,swap]{$p_{1}^{\ell}$}(2);
          \draw[a](1)to node[la]{$p_{2}^{\ell}$}(4);
        \end{tikzpicture}
      \end{gathered}
    \right\}\cong\left\{
      \begin{gathered}
        \textrm{\,triple product cones\,}\\
        \begin{tikzpicture}
          \node(1)[]{$P$};
          \node(2)[below left= of 1]{$A$};
          \node(3)[below= of 1]{$B$};
          \node(4)[below right= of 1]{$C$};
          \draw[a](1)to node[la,swap]{$p$}(2);
          \draw[a](1)to node[la,swap]{$q$}(3);
          \draw[a](1)to node[la]{$r$}(4);
        \end{tikzpicture}
      \end{gathered}
    \right\}
  \]
\end{lemma}

\begin{proof}
  Given a product cone $(p_{1}^{\ell},p_{2}^{\ell})$ as above-left, the reader may readily verify that the cone $(p_{1}p_{1}^{\ell},\ p_{2}p_{1}^{\ell},\ p_{2}^{\ell})$ displays $P$ as a triple product.

  In the other direction, given a triple product cone $(p,q,r)$ let us factor the cone $(p,q)$ over $A\times B$ through the product cone $(p_{1}, p_{2})$ via $p_{1}^{\ell}\colon P \to A\times B$ which is unique among those maps satisfying $p_{1}p_{1}^{\ell}=p$ and $p_{2}p_{1}^{\ell}=q$. It is then straightforward to verify that the cone $(p_{1}^{\ell},r)$ displays $P$ as a product of $A\times B$ and $C$. By construction this assignment is inverse to the previous one.
\end{proof}

A straightforward modification of the proof above yields the dexterous version of this lemma.

\begin{lemma}\label{lemma:paste_dexter}
    Given objects $A,B,C\in\ob\sC$ and a product cone $\prodcone B {p_{1}} {B\times C}{p_{2}} C$, composition with this cone gives a bijection between the following sets.
    \[
      \left\{
      \begin{gathered}
        \textrm{\,triple product cones\,}\\
        \begin{tikzpicture}
          \node(1)[]{$P$};
          \node(2)[below left= of 1]{$A$};
          \node(3)[below= of 1]{$B$};
          \node(4)[below right= of 1]{$C$};
          \draw[a](1)to node[la,swap]{$p$}(2);
          \draw[a](1)to node[la,swap]{$q$}(3);
          \draw[a](1)to node[la]{$r$}(4);
        \end{tikzpicture}
      \end{gathered}
    \right\}\cong
    \left\{
      \begin{gathered}
        \textrm{\,product cones\,}\\
        \begin{tikzpicture}
          \node(1)[]{$P$};
          \node(2)[below left= of 1,anchor=north]{$A$};
          \node(4)[below right= of 1,anchor=north]{$B\times C$};
          \draw[a](1)to node[la,swap]{$p_{1}^{r}$}(2);
          \draw[a](1)to node[la]{$p_{2}^{r}$}(4);
        \end{tikzpicture}
      \end{gathered}
    \right\}
  \]\vspace{-2\baselineskip} 

  \hfill$\blacksquare$
\end{lemma}

These lemmas enable us to prove that every cartesian monoidal structure may be constructed to have at least a single, arbitrary component of its associator as strict without issue.

\begin{corollary}\label{cor:consistent}
  In constructing a cartesian monoidal structure on a category \sC, for a single but arbitrary triple of objects $A,B,C\in\ob\sC$ it is always possible to arrange for the associator component $\alpha_{A,B,C}\colon(A\otimes B)\otimes C\to A\otimes(B\otimes C)$ to be the identity.
\end{corollary}

\begin{proof}
  We give the cartesian monoidal structure from $A$, $B$, and $C$ ``outward''. First fix product cones $\prodcone{A}{p_{1}}{A\times B}{p_{2}}B$ and $\prodcone{B}{q_{1}}{B\times C}{q_{2}}C$, as well as any triple product $P$ of $A,B,C$ and its cone $(p,q,r)$.

  From these data, let us apply \cref{lemma:paste_sinister,lemma:paste_dexter} to the triple product cone $(p,q,r)$ and the binary cones above to derive new cones for $P$ over the pairs $A\times B,C$ and $A,B\times C$. With this, we define the values of of the tensor $\otimes\colon\sC\times\sC\to\sC$ and matching cones as follows.
  \[
    \begin{aligned}
      A\otimes B&\coloneqq A\times B \textrm{, cone } (p_{1},p_{2})\\
      B\otimes C&\coloneqq B\times C\textrm{, cone } (q_{1},q_{2})
    \end{aligned}\qquad\text{and}\qquad
    \begin{aligned}
      (A\otimes B)\otimes C&\coloneqq P\textrm{, cone } (p^{\ell}_{1},p^{\ell}_{2})\\
      A\otimes(B\otimes C)&\coloneqq P\textrm{, cone } (p^{r}_{1},p^{r}_{2})
    \end{aligned}
  \]
  In the usual construction \eqref{eqn:generic_alpha} of the associator component $\alpha_{A,B,C}\colon P\to P$ we paste the cones displayed on the left with their matching partners displayed on the right to derive triple product cones. But, by \cref{lemma:paste_sinister,lemma:paste_dexter}, this operation is inverse to our construction of the cones $(p_{1}^{\ell},p_{2}^{\ell})$ and $(p_{1}^{r},p_{2}^{r})$ and so in both cases we find the same cone, $(p,q,r)$. Thus, by construction, $\alpha_{A,B,C}=\id_{P}$. From here we may proceed with the usual construction of the cartesian monoidal structure.
\end{proof}

In the special case of Isbell's hypotheses, that is, in the presence of an object $C=C\times C$ with given cone $(p_{1},p_{2})$ we may apply the above corollary to see that $\alpha_{C,C,C}=\id_{C}$ is consistent \emph{under a certain choice of cones}. Such a particular choice is displayed in \eqref{eqn:unbiased_isbell}, provided that we derive the cones $(p^{\ell}_{1},p^{\ell}_{2})$ and $(p^{r}_{1},p^{r}_{2})$ as in the above corollary: beginning from a fixed triple product cone $C=C\times C\times C$ -- for example, $(p_{1}p_{1},p_{2}p_{1},p_{2})$.

\section{Conclusion}\label{sec:conclusion}

Isbell's argument quoted above and reproduced widely is often summarised as the assertion that strictification of the components of the associator is not possible due to implications of naturality. However, as we have seen in \cref{sec:differentcones}, this is not the case: \cref{cor:consistent} demonstrates that we can \emph{always} arrange for a cartesian monoidal structure to have a strict associator at a single but arbitrary triple of objects. As such, even for objects $C=C\times C$, there cannot be a general theorem giving an obstruction to the strictification of the component of $\alpha$ at $C$ in an arbitrary cartesian monoidal category. Isbell's lesson in this, properly interpreted in the vein of our \cref{lemma:subterm}, is rather that choosing all four cones involved in $\alpha_{C,C,C}$ to be the same and taking $\alpha_{C,C,C}=\id_{C}$ requires $\sC(-,C)$ to be subterminal.

\bibliographystyle{alpha}
\bibliography{refs}

\end{document}